\documentclass[12pt, a4paper]{amsart}
\usepackage{amscd,amsthm,amsfonts,amssymb,amsmath,euscript,mathrsfs}
\usepackage[dvips]{graphicx}
\usepackage[dvips]{graphics}
\usepackage[matrix,arrow]{xy}
\usepackage{epsfig}
\usepackage{setspace}
\usepackage{longtable}
\usepackage{url}
\usepackage{textcomp}

\newcounter{statements}

\newcounter{conjectures}

\newtheorem{theorem}[statements]{Theorem}
\newtheorem{conjecture}[conjectures]{Conjecture}

\newtheorem{lemma}[statements]{Lemma}

\newtheorem{corollary}[statements]{Corollary}
\newtheorem{proposition}[statements]{Proposition}

\theoremstyle{definition}

\newtheorem{definition}[statements]{Definition}

\theoremstyle{remark}
\newtheorem{remark}[statements]{Remark}

\newcommand{\QQ}{{\mathbb Q}}
\newcommand{\ZZ}{{\mathbb Z}}

\newcommand{\PP}{{\mathbb P}}
\newcommand{\CC}{{\mathbb C}}

\newcommand{\Aff}{{\mathbb A}}

\newcommand{\bbA}{{\mathbb A}}
\newcommand{\bbC}{{\mathbb C}}
\newcommand{\bbH}{{\mathbb H}}

\newcommand{\bbP}{{\mathbb P}}

\newcommand{\pic}{\mathrm{Pic}\,}

\newcommand{\cO}{\mathcal{O}}

\newcommand{\id}{\mathrm{id}}

\newcommand{\coh}{\operatorname{coh}\,}
\newcommand{\logdiv}{\operatorname{log}\,}
\newcommand{\crit}{\operatorname{crit}}

\begin{document}

\title{Landau--Ginzburg Hodge numbers for mirrors of del Pezzo surfaces}

\author[V.\,Lunts, V.\,Przyjalkowski]{Valery Lunts and Victor Przyjalkowski}

\address{\emph{Valery Lunts}
\newline
\textnormal{Department of Mathematics, Indiana University,
\newline
Rawles Hall, 831 East 3rd Street,
Bloomington, IN 47405, USA.
}
\newline
\textnormal{National Research University Higher School of Economics, Laboratory of Mirror Symmetry, NRU HSE, 6 Usacheva street, Moscow, 117312, Russia.}
\newline
\textnormal{\texttt{vlunts@indiana.edu}}}

\address{\emph{Victor Przyjalkowski}
\newline
\textnormal{Steklov Mathematical Institute of RAS, 8 Gubkina street, Moscow 119991, Russia.
}
\newline
\textnormal{National Research University Higher School of Economics, Laboratory of Mirror Symmetry, NRU HSE, 6 Usacheva street, Moscow, 117312, Russia.
}
\newline
\textnormal{\texttt{victorprz@mi.ras.ru, victorprz@gmail.com}}}

\maketitle

\begin{abstract}
We consider the conjectures from 
\cite{KKP14} about Landau--Ginzburg Hodge numbers associated to tamely compactifiable Landau--Ginzburg models. We test these conjectures
in case of dimension two, verifying some and giving a counterexample to the other.
\end{abstract}

\section{Introduction}
Homological Mirror Symmetry (HMS) conjecture of Kontsevich is a master conjecture which is expected to have many numerical consequences.
Such numerical predictions can be formulated and tested in cases where HMS has not yet been established. In the paper~\cite{KKP14} (Def. 3.1,3.2,3.4)
the authors define three types of Hodge-theoretical invariants
$$
f^{p,q}(Y,w),\ \ h^{p,q}(Y,w),\ \ i^{p,q}(Y,w)
$$
of tamely compactifiable (see Definitions \ref{def-1} and \ref{def-3}) Landau--Ginzburg models $w\colon Y\to \CC$.
The numbers $f^{p,q}(Y,w)$ come from the sheaf cohomology of certain logarithmic forms,
the numbers $h^{p,q}(Y,w)$ come from mirror considerations, and the numbers $i^{p,q}(Y,w)$ come from ordinary mixed Hodge theory.

It is proved in~\cite{KKP14} (p.39) that these numbers satisfy the identities
\begin{equation} \label{sum-of-hodge-intro}
\dim H^m(Y,Y_b;\bbC)=\sum _{p+q=m}i^{p,q}(Y,w)=\sum _{p+q=m}h^{p,q}(Y,w)=\sum _{p+q=m}f^{p,q}(Y,w),
\end{equation}
where $Y_b$ is a smooth fiber of $w$. And the following refinements of~\eqref{sum-of-hodge-intro} is conjectured based on HMS considerations:
\begin{equation}\label{eq-hodge-numb}
f^{p,q}(Y,w)=h^{p,q}(Y,w)=i^{p,q}(Y,w).
\end{equation}

Tamely compactifiable Landau--Ginzburg models $(Y,w)$ typically appear as mirrors of
projective Fano manifolds $X$.
HMS predicts an equivalence of triangulated categories
$$
D^b(\coh X)\simeq FS(Y,w,\omega_Y),
$$
where $D^b(\coh X)$ is a bounded derived category of coherent sheaves on $X$ and $FS(Y,w,\omega_Y)$ is a Fukaya--Seidel
category of $(Y,w)$ for an appropriately chosen symplectic form $\omega_Y$. In this situation the authors of ~\cite{KKP14} make an additional
conjecture
\begin{equation}
\label{equation: Hodge rotation}
f^{p,q}(Y,w)=h^{p,n-q}(X),
\end{equation}
where $n=\dim X=\dim Y$.

In this paper we test conjectures~\eqref{eq-hodge-numb} and~\eqref{equation: Hodge rotation} in the case $\dim Y=2$, i.\,e.
when $Y$ is a specific rational surface with a map $w\colon Y\to \CC$ such that the generic fiber is an elliptic curve. In this case we prove the equality $f^{p,q}(Y,w)=h^{p,q}(Y,w)$ and give an example, where~$i^{p,q}(Y,w)\neq h^{p,q}(Y,w)$. It is interesting to find a ``correct''
definition of numbers~$i^{p,q}(Y,w)$ which would be compatible with conjecture~\eqref{eq-hodge-numb}.
Moreover, in case the Landau--Ginzburg model is mirror to a del Pezzo surface $X$ (see~\cite{AKO06}) we prove that $f^{p,q}(Y,w)=h^{p,2-q}(X)$,
thus verifying conjecture~\eqref{equation: Hodge rotation}.

Actually we first correct slightly the definition of the numbers $h^{p,q}(Y,w)$ since the original definition in~\cite{KKP14}
is clearly not what the authors had in mind. We hope that our methods can be used in testing conjectures~\eqref{sum-of-hodge-intro}
and~\eqref{equation: Hodge rotation} in higher dimensions.

\medskip

The paper is organized as follows. In Section~\ref{section-definition} we give the main definitions
and formulate the conjectures that we consider following~\cite{KKP14}. In particular we recall definitions of the numbers $f^{p,q}(Y,w)$, $h^{p,q}(Y,w)$,
and $i^{p,q}(Y,w)$. We also formulate the main theorem of the paper (Theorem~\ref{theorem:main}).
In Section~\ref{section:monodromy} we study monodromy of Landau--Ginzburg models.
In the following sections we consider the case of dimension $2$.
In Section~\ref{section:topology} we study topology and cohomological properties of the elliptic surfaces that we are interested in.
In Section~\ref{section:LG-Hodge-numbers-for-surfaces} we compute Landau--Ginzburg Hodge numbers for the elliptic surfaces
and prove Proposition~\ref{proposition:h-numbers} and Proposition~\ref{proposition:f-numbers}, which (together with Remark \ref{remark: Hodge for del Pezzo}) give a proof
of Theorem~\ref{theorem:main}. A big part of this section is computations of $f$-adopted log forms that are needed for the numbers $f^{p,q}(Y,w)$. Finally in Section~\ref{section:main} we discuss our results. In particular we discuss a counterexample
to a part of conjectures from~\cite{KKP14} related to numbers $i^{p,q}(Y,w)$.

\medskip

\textbf{Notation and conventions.}
All varieties are defined over the field of complex numbers~$\CC$ and we consider them as topological spaces with the classical analytic topology. For a pair of topological spaces  $Y_0\subset Y$
symbols like $H^\bullet(Y)$, $H^\bullet (Y,Y_0)$ will denote the singular cohomology (resp. relative singular cohomology) with coefficients in $\bbC$.
Symbols like $H_c^\bullet (Y)$ will denote the cohomology with compact supports (of $Y$ with coefficients in the constant sheaf $\bbC _Y$).

\medskip

{\bf Acknowledgments.} The authors are grateful to V.\,Golyshev, A.\,Kasprzyk, L.\,Katzarkov, and D.\,Orlov for useful discussions.
Special thanks to T.\,Pantev and V.\,Turaev for their help.
Both authors are partially supported by Laboratory of Mirror Symmetry NRU HSE, RF Government grant, ag. \textnumero 14.641.31.0001.
V.\,Lunts is partially supported by the NSA grant H98230-15-1-0255;
V.\,Przyjalkowski 
is partially supported by 
the Program of the Presidium of
 the Russian Academy of Sciences ``Fundamental Mathematics and
 its Applications'' under grant PRAS-18-01 and by Young Russian Mathematics award.

\section{Compactified Landau--Ginzburg models and Hodge-theoretical conjectures}\label{section-definition}

Let us recall some numerical conjectures from~\cite{KKP14} which are supposed
to follow from the conjectural Homological Mirror Symmetry between Fano manifolds and Landau--Ginzburg models.

\begin{definition} \label{def-1}\emph{A Landau--Ginzburg model} is a pair $(Y,w)$, where

\begin{enumerate}
\item $Y$ is a smooth complex quasi-projective variety with trivial canonical bundle $K_Y$;

\item $w\colon Y\to \bbA ^1$ is a morphism with a compact critical locus $\crit(w)\subset Y$.
\end{enumerate}

\end{definition}

\begin{remark}
Note that there are no conditions on singularities of fibers.
\end{remark}

Following~\cite{KKP14} we assume that there exists a {\it tame} compactification of the Landau--Ginzburg model as defined below,
see also a notion of log Calabi--Yau compactification of toric Landau--Ginzburg in~\cite{Prz16} and~\cite{Prz17}.
.

\begin{definition} \label{def-3}\emph{A tame compactified Landau--Ginzburg model} is the data $((Z,f),D_Z)$, where

\begin{enumerate}
\item $Z$ is a smooth projective variety and $f\colon Z\to \bbP ^1$ is a flat morphism.

\item $D_Z=(\cup _i D^h_i)\cup (\cup _jD_j^v)$ is a reduced normal crossings divisor such that

\begin{itemize}
\item[(i)] $D^v=\cup _jD^v_j$ is a scheme-theoretical pole divisor of $f$, i.e. $f^{-1}(\infty)=D^v$. In particular $ord _{D^v_j}(f)=-1$ for all $j$;

\item[(ii)] each component $D_i^h$ of $D^h=\cup _iD^h_i$ is smooth and horizontal for $f$, i.e. $f\vert _{D^h_i}$ is a flat morphism;

\item[(iii)] The critical locus $\crit(f)\subset Z$ does not intersect $D^h$.

\end{itemize}

\item $D_Z$ is an anticanonical divisor on $Z$.

\noindent One says that $((Z,f),D_Z)$ is \emph{a compactification of the Landau--Ginzburg model}
$(Y,w)$ if in addition the following holds:

\item $Y=Z\setminus D_Z$, $f\vert _Y=w$.
\end{enumerate}

\end{definition}

\begin{remark} In~\cite{KKP14} the authors require in above definitions an additional choice of
compatible holomorphic volume forms on $Z$ and $Y$. Since these forms will play no role
in our work we omitted them.
\end{remark}

Assume that we are given a Landau--Ginzburg model $(Y,w)$ with a tame compactification~$((Z,f),D_Z)$ as above.
We denote by $n=\dim Y=\dim Z$ the (complex) dimension of~$Y$ and $Z$. Choose a point $b\in  \bbA ^1$ which is near $\infty$ and such that the fiber $Y_b=w^{-1}(b)\subset Y$ is smooth.
In~\cite{KKP14} the authors define geometrically three sets of what they call ``Hodge numbers''  $i^{p,q}(Y,w)$, $h^{p,q}(Y,w)$, $f^{p,q}(Y,w)$. Let us recall the definitions.

\subsection{The numbers $f^{p,q}(Y,w)$}\label{subs-def-fpq} Recall the definition of the logarithmic de Rham complex~$\Omega^\bullet _Z(\logdiv  D_Z)$.
Namely, $\Omega ^s _Z(\logdiv D_Z)=\wedge ^s \Omega^1 _Z(\logdiv D_Z )$ and $\Omega^1 _Z(\logdiv D_Z )$ is a locally free~$\mbox{$\cO _Z$-module}$ generated locally by
$$\frac{dz_1}{z_1},\ldots,\frac{dz_k}{z_k},dz_{k+1},\ldots,dz_n$$
if $z_1\cdot \ldots\cdot z_k=0$ is a local equation of the divisor $D_Z$. Hence in particular $\Omega ^0 _Z(\logdiv D_Z)=\cO _Z$.

The numbers $f^{p,q}(Y,w)$ are defined using the  subcomplex $\Omega^\bullet _Z(\logdiv D_Z ,f)\subset \Omega^\bullet _Z(\logdiv D_Z)$ of~$f$-{\it adapted forms}, which we recall next.

\begin{definition} For each $a\geq 0$ define \emph{a sheaf $\Omega ^a _Z(\logdiv D_Z ,f)$ of $f$-adapted logarithmic forms} as a subsheaf of $\Omega ^a _Z(\logdiv D_Z)$ consisting of forms which stay logarithmic after multiplication by $df$. Thus
$$\Omega ^a _Z(\logdiv D_Z ,f)=\{\alpha \in \Omega ^a _Z(\logdiv D_Z)\ \vert \ df\wedge \alpha \in \Omega ^{a+1} _Z(\logdiv D_Z )\},$$
where one considers $f$ as a meromorphic function on $Z$ and $df$ is viewed as a meromorphic~$\mbox{1-form}$.
\end{definition}

\begin{definition} \emph{The Landau--Ginzburg Hodge numbers} $f^{p,q}(Y,w)$ are defined as follows:
$$f^{p,q}(Y,w)=\dim H^p(Z,\Omega ^q _Z(\logdiv D_Z ,f)).$$
\end{definition}

\subsection{The numbers $h^{p,q}(Y,w)$} Let $N\colon V\to V$ be a nilpotent operator on a finite dimensional vector space $V$ such that $N^{m+1}=0$. Recall that this data defines a canonical (monodromy) {\it weight filtration centered at $m$}, $W=W_\bullet (N,m)$ of $V$
$$0\subset W_0(N,w)\subset W_1(N,w)\subset \ldots\subset W_{2m-1}(N,m)\subset W_{2m}(N,m)=V$$
with the properties
\begin{enumerate}
\item $N(W_i)\subset W_{i-2}$,
\item the map $N^l\colon gr ^{W,m}_{m+l}V\to gr ^{W,m}_{m-l}V$ is an isomorphism for all $l\geq 0$.
\end{enumerate}

Let $S^1\simeq C\subset \bbP ^1$ be a smooth loop passing through the point $b$ that goes once around $\infty$ in the counter clockwise direction in such a way that there are no singular points of $w$ on or inside~$C$. It gives
the monodromy transformation
\begin{equation*} \label{monodromy-1}T\colon H^\bullet (Y_b)\to H^\bullet(Y_b)
\end{equation*}
and also the corresponding monodromy transformation on the relative cohomology
\begin{equation}\label{monodromy-2} T\colon H^\bullet (Y,Y_b)\to H^\bullet (Y,Y_b)
\end{equation}
in such a way that the sequence
\begin{equation*}\label{long-exact-seq}
\ldots\to H^m(Y,Y_b)\to H^m(Y)\to H^m(Y_b)\to H^{m+1}(Y,Y_b)\to \ldots
\end{equation*}
 is $T$-equivariant, where $T$ acts
trivially on $H^\bullet (Y)$. (See Section~\ref{section:monodromy} for the construction and the discussion of the
monodromy transformation $T\colon H^\bullet (Y,Y_b)\to H^\bullet (Y,Y_b)$.) Since we assume that the infinite fiber $f^{-1}(\infty)\subset Z$ is a reduced divisor with normal crossings, by Griffiths--Landman--Grothendieck Theorem  (see~\cite{Ka70}) the operator $T\colon H^m (Y_b)\to H^m (Y_b)$ is unipotent and $\left(T-\id\right)^{m+1}=0$. It follows that the transformation \eqref{monodromy-2} is also unipotent. Denote by $N$ the logarithm of the transformation \eqref{monodromy-2}, which is therefore a nilpotent operator on $H^\bullet (Y,Y_b)$.
One has $N^{m+1}=0$.

\begin{definition} \label{def-Fano-type} We say that the Landau--Ginzburg model $(Y,w)$ is {\it of Fano type} if the operator~$N$ on the relative cohomology $H^{n+a} (Y,Y_b)$ has the following properties:
\begin{enumerate}
\item $N^{n-|a|}\neq 0$,

\item $N^{n-|a|+1}=0$.
\end{enumerate}
\end{definition}

The above definition is motivated by the expectation that
the Landau--Ginzburg model of Fano type usually appears as a mirror of a projective Fano
manifold 
(see Subsection \ref{explanation} below).

\begin{definition}
\label{definition:hpq} Assume that $(Y,w)$ is a Landau--Ginzburg model of Fano type. Consider the relative cohomology $H^\bullet(Y,Y_b)$ with the nilpotent operator $N$ and the induced canonical filtration $W$. \emph{The Landau--Ginzburg Hodge numbers} $h^{p,q}(Y,w)$ are defined 
 as follows:
$$h^{p,n-q}(Y,w)=\dim gr _{2(n-p)}^{W,n-a}H^{n+p-q}(Y,Y_b)\ \ \text{if $a=p-q\geq 0$},$$
$$h^{p,n-q}(Y,w)=\dim gr _{2(n-q)}^{W,n+a}H^{n+p-q}(Y,Y_b)\ \ \text{if $a=p-q< 0$}.$$

\end{definition}

\begin{remark}
\label{remark:correction}
Our Definition~\ref{definition:hpq} differs from~\cite[Definition~3.2]{KKP14}
\begin{equation}\label{wrong-def}h^{p,q}(Y,w)=\dim gr _{p}^{W,p+q}H^{p+q}(Y,Y_b)
\end{equation}
by the indices of the grading. The equation \eqref{wrong-def} seems not to be what the authors had in mind. For example according to \eqref{wrong-def} the index $p$ is allowed to vary from $0$ to $2n$ and $q$ is allowed to be negative (see Subsection
\ref{explanation} for an explanation). Also the exponent of the action of $N$ on $H^{p+q}(Y,Y_b)$ is in general not expected to be $p+q$ (see Definition \ref{def-Fano-type}).

It is implicitly conjectured in~\cite{KKP14} that all the Jordan blocks of the nilpotent operator $N$ on $H^m(Y,Y_b)$ have the same parity (see Remark \ref{expl} below). Under this assumption we do not loose any information by considering only the even subquotients of the weight filtration as in the above definition.
\end{remark}

\subsection{The numbers $i^{p,q}(Y,w)$} Recall that for each $\lambda \in \bbA ^1$ one has the corresponding sheaf~$\phi _{w-\lambda}\bbC _Y$ of vanishing cycles for the fiber $Y_\lambda$, see~\cite{De73}. The sheaf $\phi _{w-\lambda}\bbC _Y$ is supported on the fiber $Y_\lambda$ and is equal to zero if $\lambda $ is not a critical value of $w$. From the works of Schmid, Steenbrink,
and Saito
it is classically known that the constructible complex $\phi _{w-\lambda}\bbC _Y$ carries a structure of a mixed Hodge module and so its hypercohomology inherits a mixed Hodge structure, see, for instance,~\cite{PS08}. For a mixed Hodge module $S$ we will denote by $i^{p,q}S$ the $(p,q)$-Hodge numbers of the $p+q$ weight graded piece $gr ^W_{p+q}S$.

\begin{definition}[{see~\cite[Definition 3.4]{KKP14} and~\cite{GKR12}}]

\begin{enumerate}
\item
Assume that the horizontal divisor $D^h\subset Z$ is empty, i.e. assume that the map $w\colon Y\to \bbA ^1$ is proper. Then \emph{the Landau--Ginzburg Hodge numbers}~$i^{p,q}(Y,w)$ are defined as follows:
$$
i^{p,q}(Y,w)=\sum _{\lambda \in \bbA ^1}\sum _ki^{p,q+k}\bbH ^{p+q-1}(Y_\lambda ,
\phi _{w-\lambda}\bbC _Y).
$$

\item
In the general case denote by $j\colon Y\hookrightarrow Z$ the open embedding and define
similarly
$$
i^{p,q}(Y,w)=\sum _{\lambda \in \bbA ^1}\sum _ki^{p,q+k}\bbH ^{p+q-1}(Y_\lambda ,
\phi _{w-\lambda}{\bf R}j_{*}\bbC _Y).
$$
\end{enumerate}
\end{definition}

\subsection{Conjectures} It is proved in~\cite{KKP14} (p.39) that for every $m$ the above numbers satisfy the equalities
\begin{equation} \label{sum-of-hodge}
\dim H^m(Y,Y_b;\bbC)=\sum _{p+q=m}i^{p,q}(Y,w)=\sum _{p+q=m}f^{p,q}(Y,w)
\end{equation}

\begin{remark}\label{expl} It is implicitly conjectured in~\cite{KKP14} that all the Jordan blocks of the nilpotent operator $N$ on $H^m(Y,Y_b)$ have the same parity.
Under this assumption the result in~\cite{KKP14} (p.39) also implies that
\begin{equation} \label{sum-of-hodge-add}
\dim H^m(Y,Y_b;\bbC)=\sum _{p+q=m}h^{p,q}(Y,w)
\end{equation}
\end{remark}

The authors state several conjectures which together refine the
equalities \eqref{sum-of-hodge} and \eqref{sum-of-hodge-add}. The next is a modification of~\cite[Conjecture $3.6$]{KKP14}, see Remark~\ref{remark:correction}.

\begin{conjecture}\label{conj-1} Assume that $(Y,w)$ is a Landau--Ginzburg model of Fano type. Then for every $p,q$ there are  equalities
\begin{equation*}\label{equat-of-indiv-hodge-numbers}
h^{p,q}(Y,w)=f^{p,q}(Y,w)=i^{p,q}(Y,w).
\end{equation*}
\end{conjecture}

The Landau--Ginzburg model $(Y,w)$ of Fano type (together with a tame compactification) typically arises as a mirror of a
projective Fano manifold $X$, $\dim X=\dim Y$.

The following is~\cite[Conjecture $3.7$]{KKP14}, see Remark~\ref{remark:correction}.

\begin{conjecture} \label{conj-2} In the above mirror situation for each $p,q$ we have the equality
$$f^{p,q}(Y,w)=h^{p,n-q}(X),$$
where $h^{p,q}(X)$'s are the usual Hodge numbers for $X$.
\end{conjecture}

\subsection{Explanation of Conjectures}\label{explanation}
We refer the interested reader to~\cite{KKP14} for a detailed description of the motivation for Conjectures \ref{conj-1} and \ref{conj-2}. Basically the motivation comes from HMS,  Hochschild homology identifications, and identification of the monodromy operator with the Serre functor. Namely, assume that the Landau--Ginzburg model $(Y,w)$ as above (together with a tame compactification) is of Fano type and is a mirror of a projective Fano manifold $X$, $\dim X=\dim Y$. Then by HMS conjecture one expects an equivalence of categories
\begin{equation}\label{eq-of-cat}D^b(\coh X)\simeq FS((Y,w),\omega_Y),
\end{equation}
where $D^b(\coh X)$ is the bounded derived category of coherent sheaves on $X$ and $FS((Y,w),\omega_Y)$ is the Fukaya--Seidel category of the Landau--Ginzburg model $(Y,w)$ with an appropriate  symplectic form $\omega_Y$. This equivalence induces for each $a$ an isomorphism of the Hochschild homology spaces
\begin{equation*}\label{eq-of-hoch-hom}
HH_a(D^b(\coh X))\simeq HH_a(FS((Y,w),\omega_Y)).
\end{equation*}
It is known that
\begin{equation}\label{eq-of-hoch-derham}HH_a(D^b(\coh X))\simeq \bigoplus _{p-q=a}H^p(X,\Omega _X^q)
\end{equation}
and it is expected that
\begin{equation}\label{eq-of-hoch-fuk}HH_a(FS((Y,w),\omega_Y))\simeq H^{n+a}(Y,Y_b).
\end{equation}
The equivalence \eqref{eq-of-cat} and isomorphisms \eqref{eq-of-hoch-derham}, \eqref{eq-of-hoch-fuk} suggest an isomorphism
\begin{equation*}\label{summary-isom}
H^{n+a}(Y,Y_b)=\bigoplus _{p-q=a}H^p(X,\Omega _X^q).
\end{equation*}
Moreover, the equivalence \eqref{eq-of-cat} identifies the Serre functors $S_X$ and $S_Y$ on the two categories. The
functor $S_X$ acts on the cohomology $H^\bullet (X)$ and the logarithm of this operator is equal (up to a sign) to
the cup-product with $c_1(K_X)$. Since $X$ is Fano, the operator $c_1(K_X)\cup (\cdot )$ is a Lefschetz operator on the
space
$$\bigoplus _{p-q=a}H^p(X,\Omega _X^q)$$
for each $a$.
On the other hand, the Serre functor $S_Y$ induces an operator on the space~$H^{n+a}(Y,Y_b)$ which is the inverse
of the monodromy transformation $T$. This suggests that the weight filtration for the nilpotent operator $c_1(K_X)\cup (\cdot )$ on
the space $\bigoplus _{p-q=a}H^p(X,\Omega _X^q)$ should coincide with the similar filtration for the logarithm $N$ of the operator $T$ on
$H^{n+a}(Y,Y_b)$.
First notice that the operator
$c_1(K_X)\cup (\cdot )$ on
the space $\bigoplus _{p-q=a}H^p(X,\Omega _X^q)$ satisfies
$(c_1(K_X)\cup (\cdot ))^{n-|a|}\neq 0$ by the Hard Lefschetz theorem and~$\left(c_1(K_X)\cup (\cdot )\right)^{n-|a|+1}= 0$. This explains our Definition \ref{def-Fano-type}.
Moreover, the induced filtration $W$ on $\bigoplus _{p-q=a}H^p(X,\Omega _X^q)$ has the properties:
$$h^{p,q}(X)=gr^{W,n-a}_{2(n-p)}\left[\bigoplus _{p-q=a}H^p(X,\Omega _X^q)\right]\ \ \text{if $a\geq 0$}$$
and
$$h^{p,q}(X)=gr^{W,n+a}_{2(n-q)}\left[\bigoplus _{p-q=a}H^p(X,\Omega _X^q)\right]\ \ \text{if $a< 0$}.$$
Thus one expects the equality of Hodge numbers
$$h^{p,n-q}(Y,w)=h^{p,q}(X),$$
which is a combination of the above conjectures.

\subsection{Summary of results} In this work we consider
tame compactified Landau--Ginzburg model $(Z,f)$ of dimension $2$. More precisely, we consider a rational elliptic surface $f\colon Z\to \bbP ^1$  with $f^{-1}(\infty)$ being a reduced divisor which is a wheel of $d$ rational curves, $1\leq d\leq 9$ (it is a nodal rational curve if $d=1$). In this case the horizontal divisor $D^h$ is empty, so $D=D^v$. In the paper \cite{AKO06} it is proved that the corresponding Landau--Ginzburg model $(Y,w)$ appears as a (homological) mirror of a del Pezzo surface $X$ of degree $d$. The authors also establish HMS for the case $d=0$: in this case $f^{-1}(\infty)$ is a smooth elliptic curve and $(Y,w)$ is mirror to the blowup $X$ of $\bbP ^2$ in $9$ points of intersection of two cubic curves. Note that such~$X$ is not Fano, hence one expects that the corresponding Landau--Ginzburg model $(Y,w)$ is not of Fano type. We confirm this prediction. The next theorem summarizes the main results of our paper.

\begin{theorem}\label{theorem:main} Let $f\colon Z\to \bbP ^1$ be an elliptic surface with the reduced infinite fiber
$D=f^{-1}(\infty)$ which is a wheel of $d$ rational curves for $1\leq d\leq 9$ or is a smooth elliptic curve for $d=0$. We assume that $f$ has a section. As before put $(Y,w)=(Z\setminus D,f\vert _{Z\setminus D})$.

\begin{enumerate}
\item[(i)] If $1\leq d\leq 9$, then the Landau--Ginzburg model $(Y,w)$ is of Fano type and there are equalities of Hodge numbers $$f^{p,q}(Y,w)=h^{p,q}(Y,w).$$

\item[(ii)] Let $1\leq d\leq 9$ and let $X$ be a del Pezzo surface
which is a mirror
in the sense of~\cite{AKO06} to the Landau--Ginzburg model $(Y,w)$. There are equalities of Hodge numbers $$f^{p,q}(Y,w)=h^{p,2-q}(X).$$

\item[(iii)] If $d=0$, then $(Y,w)$ is not of Fano type.
\end{enumerate}
\end{theorem}

The proof of Theorem \ref{theorem:main} is contained in~Proposition~\ref{proposition:h-numbers}, Proposition~\ref{proposition:f-numbers}, and Remark~\ref{remark: Hodge for del Pezzo}.

Thus Conjecture \ref{conj-1} about the numbers $f^{p,q}(Y,w)$, $h^{p,q}(Y,w)$ and Conjecture~\ref{conj-2} hold in case~$(Y,w)$ is of Fano type ($1\leq d\leq 9$). We will also show that in the context of Theorem~\ref{theorem:main} the numbers
$i^{p,q}(Y,w)$ are {\it not} equal to the numbers $f^{p,q}(Y,w)$ (or to the numbers
$h^{p,q}(Y,w)$, or $h^{p,2-q}(X)$), therefore providing a
counter example to Conjecture \ref{conj-1}, see Remark~\ref{remark:i-numbers}.
We do not know how to define the ``correct'' numbers $i^{p,q}(Y,w)$, which would make Conjecture~\ref{conj-1} true.

\section{Monodromy action on relative cohomology}
\label{section:monodromy}
Let $V$ be a smooth complex algebraic variety
of dimension $n$ with a proper morphism~$w\colon V\to \bbC$. Let $b \in \bbC $ be a regular value of $w$.
In this section we construct the monodromy action on the relative homology $H_\bullet (V,V_b)$, which by duality will induce the desired action on~$H^\bullet (V,V_b)$.

Let $C\simeq S^1\subset \bbP ^1$ be a smooth loop passing through the point $b$ that goes once around the $\infty$ in the counter clockwise direction in such a way that there are no singular values of $w$ on or inside $C$. Denote by $M$ the preimage $w^{-1}(C)\subset Y$. Then $M$ is a compact oriented smooth manifold which contains the fiber $V_b$.
The (real) dimensions of $M$ and $V_b$ are $2n-1$ and~$2n-2$ respectively.
By Ehresmann's Lemma the map $w\colon M\to C$ is a locally trivial fibration of smooth manifolds with the fibers diffeomorphic to $V_b$.
Hence there exists a diffeomorphism~$T\colon V_b\to V_b$ such that $M$ is diffeomorphic to the quotient
$$M=V_b\times [0,1]/\{ (a,0)=(T(a),1)\ \text{for all $a\in V_b$}\}.$$
For the pair $(M,V_b)$ we have the corresponding long exact homology sequence
\begin{equation}\label{les}\ldots\to H_i(V_b)\stackrel{\alpha _i}{\to} H_i(M)\stackrel{\beta _i}{\to} H_i(M,V_b)\stackrel{\partial _i}{\to} H_{i-1}(V_b)\to \ldots
\end{equation}
The diffeomorphism $T\colon V_b\to V_b$ induces an automorphism~$T\colon H_i(V_b)\to H_i(V_b)$ for each $i$.

\begin{lemma} \label{lemma-formula} For each $i\geq 0$, there exists a homomorphism $L_i\colon H_i(V_b)\to H_{i+1}(M,V_b)$ such that for all $x\in H_i(V_b)$ we have
$$\partial _{i+1}L_i(x)=T(x)-x.$$
\end{lemma}

\begin{proof} Let $z$ be an $i$-dimensional cycle in $V_b$. Consider the $(i+1)$-dimensional relative cycle~$z\times [0,1]$ in $(V_b\times [0,1],V_b\times \{0\}\cup V_b\times \{1\})$ with boundary $z\times \{1\}-z\times \{0\}$.
Its image~$L_i(z)$ in $M$ is a relative $(i+1)$-cycle with boundary $T(z)-z$ in $V_b$.
This construction yields the required homomorphism $L_i\colon H_i(V_b)\to H_{i+1}(M,V_b)$. Given $x\in H_i(V_b)$ the
assertion of the lemma
is clear from the construction.
\end{proof}

\begin{proposition}\label{injective} For each $i\geq 0$ the map $L_i\colon H_i(V_b)\to H_{i+1}(M,V_b)$ is injective.
\end{proposition}

\begin{proof} Let $z$ be an $i$-cycle on $V_b$, which represents a nonzero homology class $[z]\in H_i(V_b)$. By Poincar\'e duality for the compact smooth manifold $V_b$ of dimension $2n-2$, there exists a $(2n-2-i)$-cycle $z^\prime $ on $V_b$ such that the homological intersection $[z^\prime]\cdot [z]$ is non-zero. Choose a fiber $V_\epsilon\subset M$ of $w$ near $V_b$. By the construction of $M$, the fibers $V_b$ and $V_\epsilon$ are canonically identified. Let $z^\prime _\epsilon$ be the $(2n-2-i)$-cycle
on $V_\epsilon$ corresponding to $z^\prime$ under this identification. We will consider $z^\prime _\epsilon$ as a $(2n-2-i)$-cycle in the open manifold $M\setminus V_b$ and let~$[z^\prime _\epsilon]\in H_{2n-2-i}(M\setminus V_b)$ be its homology class. We have the perfect Lefschetz  duality pairing (see~\cite[Theorem 6.2.19]{Sp81})
$$H_{2n-2-i}(M\setminus V_b)\times H_{i+1}(M,V_b)\to \bbC$$
defined by intersection of cycles.
By the construction there is an equality of intersection numbers
$$[z^\prime _\epsilon]\cdot L_i[z]=\pm [z^\prime ]\cdot [z]\neq 0.$$
Hence $L_i[z]\neq 0$.
\end{proof}

\begin{definition} For each $i$ define the endomorphism $T\colon H_{i}(M,V_b)\to H_{i}(M,V_b)$ as~$T=\id +L_{i-1}\partial _{i}$ and the endomorphism $T\colon H_i(M)\to H_i(M)$ as $T=\id $.
(In particular~$T=\id $ on $H_0(M,V_b)$.)
\end{definition}

\begin{corollary}
\begin{itemize}
\item[(i)] For each $i$, the image of $\beta _i$ is equal to the space of $T$-invariants~$H_i(M,V_b)^T$.

\item[(ii)] For each $i$, the kernel of the map $\alpha _i\colon H_i(V_b)\to H_i(M)$ contains the subspace~$(T-\id)H_i(V_b)$.
Hence $\alpha _i$ factors through the space of coinvariants $H_i(V_b)_T$.

\item[(iii)] The long exact sequence \eqref{les} is compatible with the endomorphisms $T$.
\end{itemize}

\end{corollary}

\begin{proof}
\begin{itemize}
\item[(i)] For $u\in H_i(M)$ we have
$$T(\beta _i(u))=\beta _i(u)+L_{i-1}\partial _i\beta _i(u)=\beta _i(u).$$
Vice versa, let $y\in H_i(M,V_b)^T$. Then $T(y)=y+L_{i-1}\partial _i(y)=y$, i.\,e.
$L_{i-1}\partial _i(y)=0$. However $L_{i-1}$ is injective by Proposition~\ref{injective}.
Hence $\partial _i(y)=0$, i.e. $y$ is in the image of $\beta _i$.

\item[(ii)] By Lemma~\ref{lemma-formula}, for $x\in H_i(V_b)$ we have
$$\partial _{i+1}L_i(x)=T(x)-x,$$
hence the image of $\partial _{i+1}$ contains the space $(T-\id)H_i(V_b)$.

\item[(iii)] The compatibility of the maps $\alpha _i$ and $\beta _i$ with $T$ follows from
(i) and (ii). Let~$y\in H_{i+1}(M,V_b)$. Then
$$\begin{array}{rcl}\partial _{i+1}T(y)& = & \partial _{i+1}\left(y+L_{i}\partial _{i+1}(y)\right)\\
& = & \partial _{i+1}(y)+\partial _{i+1}L_i\partial _{i+1}(y)\\
& = & \partial _{i+1}(y)+ (T-\id )\partial _{i+1}(y)\\
& = & T\partial _{i+1}(y).
\end{array}
$$
This proves the corollary.
\end{itemize}
\end{proof}

The inclusion of the pairs $(M,V_b)\subset (V,V_b)$ induces a morphism of the
homology sequences
$$\begin{array}{ccccccc}

\ldots \to & H_i(M) & \to & H_i(M,V_b)& \stackrel{\partial _i}{\to } & H_{i-1}(V_b) & \to \ldots\\
         & \downarrow & & \downarrow \gamma _i &  & \parallel & \\
\ldots\to & H_i(V) & \to & H_i(V,V_b) & \stackrel{\partial _i}{\to } & H_{i-1}(V_b)  & \to  \ldots
\end{array}
$$

\begin{definition} Let us define for each $i\geq 0$ the endomorphism $T\colon H_i(V,V_b)\to H_i(V,V_b)$ as the composition
$$T(y)=y+\gamma _iL_{i-1}\partial _i(y)$$
for $y\in H_i(V,V_b)$. In particular, $T=\id $ on $H_0(V,V_b)$. We also define $T\colon H_i(V)\to H_i(V)$ to be the identity.

By duality this defines the operators $T$ on the cohomology $H^i(V_b)$, $H^i(V,V_b)$, $H^i(V)$.
\end{definition}

\begin{corollary} The sequence
\begin{equation*}\label{seq2}
\ldots\to H_i(V) \to  H_i(V,V_b) \to   H_{i-1}(V_b)   \to  \ldots
\end{equation*}
is compatible with the endomorphisms $T$. Hence also the dual cohomology sequence
\begin{equation*}\label{seq3}
\ldots\to H^{i-1}(V_b) \to  H^i(V,V_b)  \to   H^i(V)   \to  \ldots
\end{equation*}
is compatible with $T$.
\end{corollary}

\begin{proof} This follows directly from the definition of the operators $T$ together with
the formula in Lemma \ref{lemma-formula}.
\end{proof}

\begin{proposition} \label{prop-sum}
\begin{itemize}
\item[(i)] Assume that the morphism $\gamma _i\colon H_i(M,V_b)\to H_i(V,V_b)$ is injective. Then the image of the morphism $H_i(V)\to H_i(V,V_b)$ is the space $H_i(V,V_b)^T$
of $T$-invariants.

\item[(ii)] If $H^{2n-i-1}(V)=0$, then the map $H_i(M,V_b)\to H_i(V,V_b)$ is injective.
Hence by (i) the image of the morphism $H_i(V)\to H_i(V,V_b)$ is the space $H_i(V,V_b)^T$
of \mbox{$T$-invariants}.
\end{itemize}
\end{proposition}

\begin{proof}
\begin{itemize}
\item[(i)]
Since $T$ acts trivially on $H_i(V)$ and the map
$H_i(V)\to H_i(V,V_b)$ is compatible with $T$, its image is contained in the
space $H_i(V,V_b)^T$. Vice versa, if $y\in H_i(V,V_b)^T$, then
$y=y+\gamma _iL_{i-1}\partial _i(y)$, i.e. $\gamma _iL_{i-1}\partial _i(y)=0$.
Hence by our assumption $L_{i-1}\partial _i(y)=0$. But $L_{i-1}$ is injective by
Proposition \ref{injective}, hence $\partial _i(y)=0$, i.e. $y$ lies in the image of the map $H_i(V)\to H_i(V,V_b)$.

\item[(ii)]
Consider the commutative diagram
$$\begin{array}{ccccccccc}
\ldots\to & H_{i}(V_b) & \to & H_i(M) & \to & H_i(M,V_b)& \stackrel{\partial _i}{\to } & H_{i-1}(V_b) & \to \ldots\\
      & \parallel   & & \downarrow \delta _i & & \downarrow \gamma _i &  & \parallel & \\
\ldots\to  &  H_{i}(V_b) & \to & H_i(V) & \to & H_i(V,V_b) & \stackrel{\partial _i}{\to } & H_{i-1}(V_b)  & \to  \ldots
\end{array}
$$
By an easy diagram chasing one sees that if $\delta _i$ is injective, then so is
$\gamma _i$.

Let $\Delta \subset \bbP ^1$ be the closed disc containing $\infty$ and bounded by the loop $C$. Let~$\Delta ^0=\Delta \setminus C$ be its interior. Finally let~$W\subset V$ be the complement of the preimage $w^{-1}(\Delta ^0)$. Then
$W$ is a compact manifold with the boundary $M$. Note that the inclusion
$W\subset V$ is a homotopy equivalence since $w$ has no singular values in~$\Delta\setminus\{\infty\}$.
Hence it suffices to prove that the map
$H_i(M)\to H_i(W)$ is injective. This map is part of the long exact sequence
$$\ldots\to H_{i+1}(W,M)\to H_i(M)\to H_i(W)\to \ldots.$$
So it suffices to prove that if $H^{2n-i-1}(W)=0$, then $H_{i+1}(W,M)=0$.
This follows from Lefschetz duality
$$H^{2n-q}(W)\simeq H_q(W,M)$$
for the compact oriented $2n$-dimensional manifold $W$ with boundary
(\cite[Theorem 6.2.20]{Sp81}).
\end{itemize}
\end{proof}

\begin{remark} Note that in the above discussion of the monodromy action on homology groups we could have considered
the homology with integral coefficients. All the results of this section are also valid in such a setup.
\end{remark}

\section{Topology of rational elliptic surfaces}
\label{section:topology}
Now we use the notation of Section~\ref{section-definition} for the special case which we will consider in the rest of the paper.
Fix a number $0\leq d\leq 9$ and let $f\colon Z\to \PP^1$ be a rational elliptic surface such that $D=D^v=f^{-1}(\infty)$ is a wheel $I_d$
of $d$ smooth rational curves for $d\geq 2$, a rational curve with one node $I_1$ for $d=1$, and a smooth elliptic curve $I_0$ for $d=0$.
Assume in addition that there exists a section $\PP^1\to E\subset Z$.
Recall that $Y=Z\setminus D$.

Since $Z$ is rational, $\chi (\cO _Z)=1$. One has $-K_Z=D$, see, for instance,~\cite[\S10.2]{ISh89}. Hence~$c_1^2(Z)=0$, so by Noether's formula the topological Euler characteristic of $Z$ is equal to $12$. This means that
$$
h^i(Z)=\left\{
                     \begin{array}{ll}
                       1, & i=0,4; \\
                       10, & i=2; \\
                       0, & \hbox{otherwise.}
                     \end{array}
                   \right.
$$
By the
adjunction formula $\left(K_Z+E\right)\cdot E=2g(E)-2=-2$, so
$E^2=-1$.

\begin{lemma}
\label{lemma:C_D}
\begin{itemize}
\item[(i)]
If $d=0$, then
$$
h^i(D)=\left\{
                     \begin{array}{ll}
                       1, & i=0,2; \\
                       2, & i=1; \\
                       0, & \hbox{otherwise.}
                     \end{array}
                   \right.
$$
\item[(ii)]
If $d>0$, then
$$
h^i(D)=\left\{
                     \begin{array}{ll}
                       1, & i=0,1; \\
                       d, & i=2; \\
                       0, & \hbox{otherwise.}
                     \end{array}
                   \right.
$$
\end{itemize}
\end{lemma}

\begin{proof}

The part (i) is clear. Prove the part (ii).
Let $p_1,\ldots p_d$ be the intersection points of the components of $D$.
Let $\pi\colon \widetilde{D}\to D$ be the normalization. Then $\widetilde{D}$ is a disjoint union of $d$ copies of $\PP^1$.
Consider an exact sequence of sheaves on $D$
\begin{equation}\label{resol-seq}
0\to \CC_{D}\to \pi_*\pi^*\CC_{D}\to \oplus_{i=1}^d \CC_{p_i}\to 0,
\end{equation}
where $\CC_{p_i}$ is a skyscraper sheaf supported at $p_i$.
Notice that
$$
\dim H^i(D,\pi_*\pi^*\CC_{D})=\dim H^i(\widetilde{D})=\left\{
                                                    \begin{array}{ll}
                                                      d, & i=0,2; \\
                                                      0, & i=1.
                                                    \end{array}
                                                  \right.
$$
Notice also that $H^0(D,\CC_{D})=\CC$ and the map $H^0(D,\CC_{D})\to H^0(D,\pi_*\pi^*\CC_{D})$ is injective.
The lemma now follows from the long exact sequence of cohomology applied to the short exact sequence~\eqref{resol-seq}.
\end{proof}

\begin{lemma}
\label{lemma:ZtoD}
The restriction map $s\colon H^2(Z)\to H^2(D)$ is surjective.
\end{lemma}

\begin{proof} Since $Z$ is a rational surface we have $NS(Z)\otimes \CC=H^2(Z)$. Also we have~$NS(D)\otimes \CC=H^2(D)$ since $H^2(D,\cO _D)=0$.
Therefore it suffices to prove that the restriction map
$$
NS(Z)\otimes\QQ\to NS(D)\otimes \QQ
$$
is surjective.

If $d$ is $0$ or $1$, the curve $D$ is irreducible, and the space $NS(D)\otimes \QQ$ is one-dimensional and is spanned by the first Chern class of any ample line bundle. So it's suffices to take an ample line bundle on $Z$ and restrict it to $D$.

So assume that $d\geq 2$, which means that $D$ is a wheel of smooth rational curves. Let $D_1,\ldots,D_d$
be its irreducible components.
First notice that there is an isomorphism~$NS(D)\cong \ZZ^d$ given by the map
$$
\pic(D)\ni \mathcal L \to (\deg \left.\mathcal L\right|_{D_1},\ldots,\deg \left.\mathcal L\right|_{D_d}).
$$
Then, for any $i$ one has
$$
-2=\deg K_{D_i}=D_i\cdot(D_i+K_Z)=D_i^2.
$$
This means that the Gram matrix $(D_i\cdot D_j)$ is equal to 
$$\left(
  \begin{array}{rrrrrr}
    -2 & 1 & 0 & 0 & \ldots & 1 \\
    1 & -2 & 1 & 0 & \ldots & 0 \\
    0 & 1 & -2 & 1 & \ldots & 0 \\
    \ldots & \ldots & \ldots & \ldots & \ldots & \ldots \\
    0 & \ldots & 0 & 1 & -2 & 1 \\
    1 & \ldots & 0 & 0 & 1 & -2 \\
  \end{array}
\right).
$$
To prove the asserted surjectivity it suffices to find divisors $F_1,\ldots,F_d$ on the surface $Z$, such that the intersection
matrix $(F_i\cdot D_j)$ is non-degenerate.
The section $E$ intersects a unique component of $D$, say $D_d$, since $E\cdot D=1$.
So we have $E\cdot D_d=1$ and $E\cdot D_i=0$ for $i\neq d$.
Taking $F_1=D_1,\ldots, F_{d-1}=D_{d-1}, F_d=E$ one gets the intersection matrix $(F_i\cdot D_j)$ is
$$\left(
  \begin{array}{rrrrrr}
    -2 & 1 & 0 & 0 & \ldots & 0 \\
    1 & -2 & 1 & 0 & \ldots & 0 \\
    0 & 1 & -2 & 1 & \ldots & 0 \\
    \ldots & \ldots & \ldots & \ldots & \ldots & \ldots \\
    0 & \ldots & 0 & 1 & -2 & 0 \\
    0 & \ldots & 0 & 0 & 0 & 1 \\
  \end{array}
\right),
$$
whose determinant is equal to $(-1)^{d-1}d$.
\end{proof}

Next we compute the cohomology $H_c^i(Y)$ of $Y$ with compact support.

\begin{lemma}
\label{lemma:C_Y}
The following equalities hold.
$$h^i_c(Y)=
h^i(Z,j_!\CC_Y)=\left\{
                     \begin{array}{ll}
                       0, & i=0,1,3; \\
                       11-d, & i=2; \\
                       1, & i=4.
                     \end{array}
                   \right.
$$
\end{lemma}

\begin{proof} The first equality follows from the fact that $Z$ is compact. For the second one
consider the short exact sequence of sheaves
$$
0\to j_!\CC_{Y}\to \CC_{Z}\to \CC_{D}\to 0.
$$

For $d=0$ by Lemma~\ref{lemma:C_D}(i) the induced long exact sequence of cohomology $H^\bullet(Z,-)$ looks like

\[\xymatrix@R=0.5pc {
 &0 \ar[r] & j_!\CC_{Y}\ar[r] & \CC_{Z}\ar[r] & \CC_{D}\ar[r] & 0\\
H^0 && 0 & \CC \ar[r]^r       & \CC &\\
H^1 && ?_1 & 0         & \CC^2 &\\
H^2 && ?_2 & \CC^{10} \ar[r]^s  & \CC &\\
H^3 && ?_3 & 0         & 0 &\\
H^4 && \CC & \CC       & 0. &\\
}\]

For $d>0$ by Lemma~\ref{lemma:C_D}(ii) the same long exact sequence looks like

\[\xymatrix@R=0.5pc {
 &0 \ar[r] & j_!\CC_{Y}\ar[r] & \CC_{Z}\ar[r] & \CC_{D}\ar[r] & 0\\
H^0 && 0 & \CC \ar[r]^r       & \CC &\\
H^1 && ?_1 & 0         & \CC &\\
H^2 && ?_2 & \CC^{10} \ar[r]^s  & \CC^d &\\
H^3 && ?_3 & 0         & 0 &\\
H^4 && \CC & \CC       & 0. &\\
}\]
The maps $r$ in both cases are obviously surjective, hence $?_1=0$.
By Lemma~\ref{lemma:ZtoD} the maps $s$ are surjective.
Hence $?_2=\CC^{11-d}$, $?_3=0$.
\end{proof}

\begin{corollary}\label{coho-of-y}
By Poincare duality for $Y$ one has
\begin{equation*}h^i(Y)=\left\{ \begin{array} {rl}  1, & \text{if $i=0$;} \\
                       11-d, & \text{if $i=2$;}\\
                        0, & \text{if $i=1,3,4$.}
                        \end{array}
                        \right.
\end{equation*}
\end{corollary}

\section{Landau--Ginzburg Hodge numbers for rational elliptic surfaces}
\label{section:LG-Hodge-numbers-for-surfaces}

\subsection{The numbers $h^{p,q}(Y,w)$}
We keep the notation of Section~\ref{section:topology}.

Consider the long exact sequence of homology
\begin{equation*} \label{seq4} \ldots\to  H_2(Y)  \to  H_2(Y,Y_b)  \to   H_{1}(Y_b)   \to  \ldots
\end{equation*}
Recall that there is a compatible action of the monodromy $T$ on each term of this sequence as explained in Section~\ref{section:monodromy}.

\begin{corollary}\label{cor-of-general} The image of the map $H_2(Y)  \to  H_2(Y,Y_b)$ coincides with the space $H_2(Y,Y_b)^T$
of $T$-invariants.
\end{corollary}

\begin{proof} In the notation of Proposition~\ref{prop-sum} we have $n=2$, $i=2$, and $H^{2n-i-1}(Y)=H^1(Y)=0$, see Corollary \ref{coho-of-y}. Hence the assertion follows from
 Proposition \ref{prop-sum}(ii).
\end{proof}

\begin{proposition}
\label{proposition:h-numbers}
\begin{enumerate}
\item[(i)]
We have \begin{equation}\label{fact-coh}H^k(Y,Y_b)=
\left\{\begin{array}{rl}
                       \CC^{12-d}, & k=2; \\
                       0, & \hbox{otherwise.}
\end{array}\right.
\end{equation}

\item[(ii)] For $d>0$ the Landau--Ginzburg model $(Y,w)$ is of Fano type and
\begin{equation}\label{eq-hpq} h^{p,q}(Y,w)=\left\{\begin{array}{rl}
                       1, & (p,q)=(0,2),(2,0); \\
                       10-d, & (p,q)=(1,1); \\
                       0, & \hbox{otherwise.}
\end{array}\right.
\end{equation}

\item[(iii)] For $d=0$ the Landau--Ginzburg model $(Y,w)$ is not of Fano type. More precisely, the $T$-action on $H^2(Y,Y_b)$ has 2 Jordan blocks of size 2 and 8 blocks of size 1.
(So no blocks of size 3).
\end{enumerate}
\end{proposition}

This proposition proves Theorem~\ref{theorem:main}(iii) and computes the right hand side of the equality of Theorem~\ref{theorem:main}(i).

The proof of the proposition will occupy the rest of this subsection.

\begin{lemma} \label{lemma-surj} The restriction map $H^2(Y)\to H^2(Y_b)$ is surjective. Hence the map~$H_2(Y_b)\to H_2(Y)$
is injective.
\end{lemma}

\begin{proof} Since $Y_b$ is a smooth projective curve, $H^2(Y_b)$ has dimension one and is spanned by the first Chern class $c_1(L)$ of any ample line bundle $L$ on $Y_b$. It suffices to take any ample line bundle $M$ on $Y$, so that its restriction
$L=M\vert _{Y_b}$ is also ample and $c_1(M)\in H^2(Y)$ restricts to $c_1(L)\in H^2(Y_b)$.
\end{proof}

The equation \eqref{fact-coh} now follows from the long exact sequence of cohomology
$$\ldots\to H^i(Y,Y_b)\to H^i(Y)\to H^i(Y_b)\to \ldots $$ using Corollary~\ref{coho-of-y}, the fact that $Y_b$ is an elliptic curve, and
Lemma \ref{lemma-surj}. This proves part (i) of the proposition.

To prove parts (ii) and (iii) it remains to understand the action of the monodromy $T$ on~$H_2(Y,Y_b)$.

Consider the part of the long exact sequence of homology
\begin{equation*}\label{piece-of-long-homology-1}
H_3(Y,Y_b)\to H_2(Y_b)\to H_2(Y)\to H_2(Y,Y_b)\to H_1(Y_b)\to H_1(Y).
\end{equation*}
We know that the map $H_2(Y_b)\to H_2(Y)$ is injective and that $H_1(Y)=H^1(Y)^\vee=0$. Hence the sequence
\begin{equation}\label{piece-of-long-homology-2}
0\to H_2(Y_b)\to H_2(Y)\to H_2(Y,Y_b)\to H_1(Y_b)\to 0
\end{equation}
is also exact. We have $H_2(Y_b)=\CC$, $H_1(Y_b)=\CC ^2$, $H_2(Y)=\CC^{11-d}$, hence the sequence
\eqref{piece-of-long-homology-2} is isomorphic to
\begin{equation*}\label{piece-of-long-homology-3}
0\to \bbC \to \bbC ^{11-d}\to \bbC ^{12-d}\to \bbC ^{2} \to 0.
\end{equation*}

These sequences are $T$-equivariant, where $T$ acts trivially on $H_2(Y_b)$ and $H_2(Y)$. By Landman's theorem
$T$ acts unipotently on $H_1(Y_b)$.

For $d=0$ the fiber $f^{-1}(\infty)$ is smooth, hence the action of $T$ on $H_1(Y_b)$ is trivial.
Therefore the exact sequence~\eqref{piece-of-long-homology-2} and Corollary \ref{cor-of-general} imply that the $T$-action on $H_2(Y,Y_b)$ is unipotent with two Jordan blocks of size $2$ and eight blocks of size $1$. This means that the Landau--Ginzburg model $(Y,w)$ is not of Fano type, which proves (iii).

For $d>0$ the fiber $f^{-1}(\infty)$ is singular, so the $T$-action on $H_1(Y_b)$ is nontrivial (see~\cite[Table $1$]{Ko63}).
Therefore the exact sequence~\eqref{piece-of-long-homology-2} and Corollary \ref{cor-of-general} imply that the $T$-action on~$H_2(Y,Y_b)$ is unipotent with one Jordan block of size $3$ and $9-d$ blocks of size $1$. Therefore~$(Y,w)$ is of Fano type and equations \eqref{eq-hpq} hold. This completes the proof of  Proposition~\ref{proposition:h-numbers}.

\subsection{The numbers $f^{p,q}(Y,w)$}
Recall that we have the open embedding $j\colon Y\hookrightarrow Z$.

\begin{lemma}
\label{lemma:log_cohomology} We have
$$\Omega^0_Z(\logdiv D)\cong\cO_Z \quad \text{and} \quad \Omega^2_Z(\logdiv D)\cong\cO_Z.$$
Hence
$$\Omega^0_Z(\logdiv D)(-D)\cong\Omega^2_Z(\logdiv D)(-D)\cong\omega _Z.$$
\end{lemma}

\begin{proof} This follows from the definition of the logarithmic complex in Subsection \ref{subs-def-fpq}
and the fact that $D$ is the anticanonical divisor.
\end{proof}

\begin{proposition}\label{prop-fpq}
The following equalities hold.
\begin{equation}\label{first-two}
h^i(Z,\Omega^0_{Z}(\logdiv D)(-D))=h^i(Z,\Omega^2_{Z}(\logdiv D)(-D))
=\left\{                                                                   \begin{array}{ll}
                                                                              0, & \hbox{i=0,1;} \\
                                                                                        1, & \hbox{i=2,}
                                                                                      \end{array}                                                                                  \right.
\end{equation}
\begin{equation}\label{last-one}
h^i(Z,\Omega^1_Z(\logdiv D)(-D))=\left\{
                     \begin{array}{ll}
                       0, & \hbox{i=0,2;} \\
                       10-d, & \hbox{i=1.}
                     \end{array}
                   \right.
\end{equation}
\end{proposition}

\begin{proof} Since the surface $Z$ is rational one has
$$
h^i(Z,\cO_Z)=\left\{
                                                                                      \begin{array}{ll}
                                                                                        1, & \hbox{i=0;} \\
                                                                                        0, & \hbox{i=1,2,}
                                                                                      \end{array}
                                                                                    \right.
$$
so by Serre duality
$$
h^i(Z,\omega  _{Z})=\left\{
                                                                                      \begin{array}{ll}
                                                                                        0, & \hbox{i=0,1;} \\
                                                                                        1, & \hbox{i=2.}
                                                                                      \end{array}
                                                                                    \right.
$$
So the equalities \eqref{first-two} follow from Lemma \ref{lemma:log_cohomology}.

To prove the equality \eqref{last-one}, notice that the complex
$$
\Omega^0_{Z}(\logdiv D)(-D)\to\Omega^1_{Z}(\logdiv D)(-D)\to \Omega^2_{Z}(\logdiv D)(-D)\to 0
$$
is a resolution of the sheaf $j_!\CC_{Y}$, see, for instance,~\cite[p. 268]{DI87}.
This gives the spectral sequence
$$E_1^{pq}=H^p(Z,\Omega^q_{Z}(\logdiv D)(-D)),$$
which converges to $H^{p+q}(Z,j_!\CC_{Y})$.
It is known that this spectral sequence  degenerates at the term $E_1$, see~\cite[Corollarie 4.2.4]{DI87}. This means that
\begin{multline*}
h^\bullet(Z,j_!\CC_{Y})=
h^\bullet(Z,\Omega^0_{Z}(\logdiv D)(-D))+\\ h^{\bullet-1}(Z,\Omega^1_{Z}(\logdiv D)(-D))
+ h^{\bullet-2}(Z,\Omega^2_{Z}(\logdiv D)(-D)).
\end{multline*}

Equalities \eqref{first-two} imply that the (numerical) $E_1$ page of this spectral sequence is as follows:
\[\xymatrix@R=0.5pc {
 1 & h^2  & 1\\
 0 & h^1  & 0\\
 0 & h^0 & 0\\
\Omega^0_{Z}(\logdiv D)(-D) & \Omega^1_{Z}(\logdiv D)(-D) & \Omega^2_{Z}(\logdiv D)(-D).\\
}\]
Therefore using Lemma \ref{lemma:C_Y} we obtain the equalities
$$
h^0(Z,\Omega^1_{Z}(\logdiv D)(-D))=0,
$$
$$
h^1(Z,\Omega^1_{Z}(\logdiv D)(-D))+1=11-d,
$$
$$ h^2(Z,\Omega^1_{Z}(\logdiv D)(-D))=0,
$$
which proves the equality \eqref{last-one}.
\end{proof}

\begin{proposition}
\label{proposition: del Pezzo f}
There are the isomorphisms
\begin{itemize}
  \item[(i)] $\Omega^0_Z(\logdiv D,f)\cong\cO_Z(-D)\cong\omega _Z;$
  \item[(ii)]
        $\Omega^2_Z(\logdiv D,f)\cong\Omega ^2_Z(\logdiv D)\cong\cO_Z.$
  \item[(iii)] There exists a short exact sequence of sheaves on $Z$
       $$0\to \Omega^1_Z(\logdiv D)(-D)\to \Omega^1_Z(\logdiv D,f)\to \cO _D\to 0.$$
\end{itemize}
\end{proposition}

\begin{proof} Let $t$ be the local coordinate
on $\bbP ^1$ at $\infty$.
Since $D=f^{-1}(\infty)$ has simple normal crossings,
locally it is equal to the zero locus of the
polynomial $xy$ on $\Aff^2$. We have
$$t=\frac{1}{f(x,y)}=xy$$
and so
$$df=d\left(\frac{1}{xy}\right)=
\frac{-1}{xy}\left(\frac{dx}{x}+\frac{dy}{y}\right).
$$

\begin{itemize}
\item[(i)] One has $\Omega^0_Z(\logdiv D)=\cO_Z$.
For any function $g\in \cO_Z$
$$df\wedge g=\frac{-g}{xy}\left(\frac{dx}{x}+\frac{dy}{y}\right).$$
So $g$ should be divisible by $xy$ to lie in $\Omega^0_Z(\logdiv D,f)$.

\item[(ii)] We have
  $$\Omega^2_Z(\logdiv D,f)=\Omega^2_Z(\logdiv D)\cong\omega_Z\otimes \omega_Z^{-1}\cong\cO_Z.$$

\item[(iii)] The proof will consist of several claims.

\medskip

\noindent{\it Claim $1$}. The inclusion $f^*\Omega ^1_{\bbP ^1}\subset \Omega ^1_Z$ induces an inclusion
$$f^*\Omega ^1_{\bbP ^1}(\logdiv \infty) \subset \Omega ^1_Z
(\logdiv D,f).$$ Indeed,
the sheaf $\Omega ^1_{\bbP ^1}(\logdiv \infty)$ is locally at $\infty$ is generated by
$$\frac{dt}{t}=\frac{d(xy)}{xy}=\frac{ydx+xdy}{xy}=
\frac{dx}{x}+\frac{dy}{y},$$
which implies that $f^*\Omega ^1_{\bbP ^1}(\logdiv \infty) \subset \Omega ^1_Z
(\logdiv D)$, and then clearly $$f^*\Omega ^1_{\bbP ^1}(\logdiv  \infty) \subset \Omega ^1_Z
(\logdiv D,f).$$

\medskip

\noindent{\it Claim $2$}. We have the inclusion $\Omega^1_Z(\logdiv D)(-D)\subset \Omega^1_Z(\logdiv D,f)$.
Indeed, a local section of $\Omega^1_Z(\logdiv D)(-D)$ near $D$ is given as $s=xy(g_1\frac{dx}{x}+g_2\frac{dy}{y})$, where
$g_1,g_2\in \cO _Z$. Then
$$df\wedge s=-(g_2-g_1)\frac{dx\wedge dy}{xy}.$$

\medskip

\noindent{\it Claim $3$.} The intersection $\Omega ^1_Z(\logdiv D)(-D)\cap f^*\Omega _{\bbP ^1}^1(\logdiv \infty)$ of subsheaves of~$\Omega ^1_Z(\logdiv D,f)$ is equal to
$f^*\Omega _{\bbP ^1}^1(\logdiv \infty)(-\infty)$.

Indeed, the sheaf $\Omega _{\bbP ^1}^1(\logdiv \infty)(-\infty)$ near $\infty$ is generated by $$t\frac{dt}{t}=xy\left(\frac{dx}{x}+\frac{dy}{y}\right)\in
\Omega ^1_Z(\logdiv D)(-D)\cap f^*\Omega _{\bbP ^1}^1(\logdiv \infty).$$
Hence
$$f^*\Omega _{\bbP ^1}^1(\logdiv \infty)(-\infty)\subset \Omega ^1_Z(\logdiv D)(-D)\cap f^*\Omega _{\bbP ^1}^1(\logdiv \infty).$$
Vice versa let
$$\eta=h(t)\frac{dt}{t}\in \Omega ^1_{\bbP ^1}(\logdiv \infty)$$
be such that
$$\eta=
h(xy)\left(\frac{dx}{x}+\frac{dy}{y}\right)\in \Omega^1(\logdiv D)(-D).$$
Then $h(xy)$ is divisible by $xy$, i.e. $h$ vanishes at $\infty$ and so $\eta\in \Omega _{\bbP ^1}^1(\logdiv \infty)(-\infty)$.

\medskip

\noindent{\it Claim $4$}. We have the equality
$$\Omega _Z^1(\logdiv D,f)=\Omega _Z^1(\logdiv D)(-D)+f^*\Omega _{\bbP ^1}(\logdiv \infty).$$
Indeed, let $\omega =g_1\frac{dx}{x}+g_2\frac{dy}{y}$, where
$g_1,g_2\in \cO _Z$, be a local section of $\Omega _Z^1(\logdiv D,f)$ near~$D$. Then $\omega =1/2(\omega _1+\omega _2)$, where
$$\omega _1 =(g_1+g_2)\frac{dx}{x}+(g_1+g_2)\frac{dy}{y}$$
and
$$\omega _2 =(g_1-g_2)\frac{dx}{x}+(g_2-g_1)\frac{dy}{y}.$$
We have
$$\omega _1=(g_1+g_2)\frac{dt}{t}\in
f^*\Omega _{\bbP ^1}(\logdiv \infty)$$
and hence by our assumption
$$df\wedge \omega _2=2(g_1-g_2)\frac{dx\wedge dy}{(xy)^2}\in \Omega ^2(\logdiv D),$$
which implies that $(g_1-g_2)$ is divisible by $xy$, i.e.
$\omega _2\in \Omega ^1_Z(\logdiv D)(-D)$. This proves Claim $4$.

We are ready to complete the proof of Proposition
\ref{proposition: del Pezzo f}.
First notice that the quotient~$\Omega ^1_{\bbP ^1}(\logdiv \infty)/\Omega ^1_{\bbP ^1}(\logdiv \infty)(-\infty)$ is isomorphic to the skyscraper sheaf $\bbC _{\infty}$.
Now it follows from Claim $3$ and Claim $4$ that there is an isomorphism of quotients
$$\Omega ^1_Z(\logdiv D,f)/\Omega ^1_Z(\logdiv D)(-D)\simeq
f^*\Omega _{\bbP ^1}(\logdiv \infty)/f^*\Omega _{\bbP ^1}(\logdiv \infty)(-\infty)=f^*\bbC _{\infty}=\cO _D,$$
which proves part (iii) of the proposition.   \qedhere
\end{itemize}
\end{proof}

\begin{proposition}\label{prop-isom-on-coh} The inclusion of sheaves $\Omega ^1_Z(\logdiv D)(-D)\subset \Omega ^1_Z(\logdiv D,f)$ induces an isomorphism of cohomology
$$H^\bullet (Z,\Omega ^1_Z(\logdiv D)(-D))\simeq
H^\bullet (Z,\Omega ^1_Z(\logdiv D,f)).$$
\end{proposition}

\begin{proof} By Proposition \ref{proposition:f-numbers}(iii) there is a short exact sequence of sheaves
 \begin{equation} \label{technical-short-exact}
 0\to \Omega^1_Z(\logdiv D)(-D)\to \Omega^1_Z(\logdiv D,f)\to \cO _D\to 0.
 \end{equation}
We have
$$h^i(Z,\cO _D)=\left\{\begin{array}{rl}
                       1, & i=0,1; \\
                       0, & \hbox{otherwise,}
\end{array}\right.
$$
and we know by Proposition \ref{prop-fpq} that
\begin{equation*}
h^i(Z,\Omega^1_Z(\logdiv D)(-D))=\left\{
                     \begin{array}{ll}
                       0, & \hbox{i=0,2;} \\
                       10-d, & \hbox{i=1.}
                     \end{array}
                   \right.
\end{equation*}
Hence the assertion of the proposition is equivalent to the nonvanishing of the boundary map
$$H^0(Z,\cO _D)\rightarrow
H^1(Z,\Omega _Z^1(\logdiv D)(-D)).$$

Let $i\colon E \hookrightarrow Z$ be the inclusion of a section of the elliptic surface $Z$ and consider the restriction of the sequence \eqref{technical-short-exact} to $E$:
\begin{equation}\label{restr-techn-short-exact}
0\to i^*\Omega^1_Z(\logdiv D)(-D)\to i^*\Omega^1_Z(\logdiv D,f)\to i^*\cO _D\to 0.
\end{equation}

Since $E$ is a section of the map $f$, it intersects
$D$ transversally (in a smooth point of~$D$).
Therefore the sequence \eqref{restr-techn-short-exact}
is also short exact.
We identify $E=\bbP ^1$; so $E\cap D=\infty$ and~$i^*\cO _D=\bbC _\infty$. The map $i^*\colon H^0(Z,\cO _D)\to H^0(E,\bbC _\infty)$ is an isomorphism, therefore it suffices to prove that the boundary map
\begin{equation}\label{boundary-map}
H^0(E,\bbC _\infty)\to H^1(E,i^*\Omega ^1_Z(\logdiv D)(-D))
\end{equation}
is not zero.

\medskip

\noindent{\it Claim.}
The sequence \eqref{restr-techn-short-exact} is isomorphic to the direct sum of short exact sequences
\begin{equation}\label{first-seq}
0\to \Omega ^1_E(\logdiv \infty)(-\infty)\to \Omega ^1_E(\logdiv \infty)\to \bbC _{\infty}\to 0
\end{equation}
and
\begin{equation*}\label{second-seq}
0\to N^*_{E/Z}(-D) \stackrel{=}{\longrightarrow} N^* _{E/Z}(-D) \to 0\to 0.
\end{equation*}

Indeed,
we have the canonical short exact sequence of vector bundles on $E$:
\begin{equation}\label{can-short-exact}
0\to N^*_{E/Z}\to i^*\Omega ^1_Z\to \Omega ^1_E\to 0
\end{equation}
with the canonical splitting $i^*f^*\colon \Omega ^1_E\to
\Omega ^1_Z$ (we identify $E=\bbP ^1$).

The sequence \eqref{can-short-exact}
induces the following commutative diagram with exact rows
\begin{equation}\label{big-diagram}
\begin{array}{ccccccccc}
0 & \to & N^*_{E/Z} & \to & i^*\Omega ^1_Z(\logdiv D) &
\to & \Omega ^1_E(\logdiv \infty) & \to & 0\\
& & \uparrow & & \uparrow & & \parallel & & \\
0 & \to & N^*_{E/Z}(-D) & \to & i^*\Omega ^1_Z(\logdiv D,f) &
\to & \Omega ^1_E(\logdiv \infty) & \to & 0\\
& & \parallel & & \uparrow & & \uparrow & & \\
0 & \to & N^*_{E/Z}(-D) & \to & i^*\Omega ^1_Z(\logdiv D)(-D) &
\to & \Omega ^1_E(\logdiv \infty)(-\infty) & \to & 0.
\end{array}
\end{equation}

Indeed, all rows are induced by the sequence \eqref{can-short-exact} and the third one is obviously exact. The vertical maps are canonical inclusions and it is clear that the whole diagram is commutative. It remains to prove the exactness of the first two rows.

There exist local coordinates $x,t$ on $Z$ near $\infty \in E$ such that
$$E=\{x=0\}\ \ \text{and}\ \ D=\{t=0\}.$$
Then $f=\frac{1}{t}$ and the sheaf $\Omega _Z^1(\logdiv D)$ (resp. $\Omega _Z^1(\logdiv D,f)$) is locally generated by $dx,\frac{dt}{t}$ (resp. by $tdx,\frac{dt}{t}$). This implies the exactness of the first two rows and proves the claim.

We can now complete the two bottom rows in the diagram \eqref{big-diagram} to a commutative diagram with exact rows and columns:
\begin{equation}\label{final-big-diagram}
\begin{array}{ccccccccc}
& & 0 & & 0 & & 0 & & \\
& & \uparrow & & \uparrow & & \uparrow & & \\
0 & \to & 0 & \to & \bbC _\infty &
\to & \bbC _\infty & \to & 0\\
& & \uparrow & & \uparrow & & \uparrow & & \\
0 & \to & N^*_{E/Z}(-D) & \to & i^*\Omega ^1_Z(\logdiv D,f) &
\to & \Omega ^1_E(\logdiv \infty) & \to & 0\\
& & \parallel & & \uparrow & & \uparrow & & \\
0 & \to & N^*_{E/Z}(-D) & \to & i^*\Omega ^1_Z(\logdiv D)(-D) &
\to & \Omega ^1_E(\logdiv \infty)(-\infty) & \to & 0\\
& & \uparrow & & \uparrow & & \uparrow & & \\
& & 0 & & 0 & & 0. & &
\end{array}
\end{equation}

Now the diagram \eqref{final-big-diagram} considered as a short exact sequence of columns splits by the maps~$i^*f^*$ as above. This proves
the claim.

Now we can complete the proof of Proposition \ref{prop-isom-on-coh}. Notice that the sequence \eqref{first-seq} is isomorphic to the
natural sequence
$$0\to \cO _E(-2)\to \cO _E(-1)\to \bbC _\infty \to 0,$$
where the boundary map $H^0(E,\bbC _\infty)\to
H^1(E,\cO _E(-2))$ is an isomorphism. Hence using the above claim we find that the boundary map
\eqref{boundary-map} is not zero, which completes the proof of Proposition \ref{prop-isom-on-coh}.
\end{proof}

\begin{proposition}
\label{proposition:f-numbers}
One has
$$f^{p,q}(Y,w)=\left\{\begin{array}{rl}
                       1, & (p,q)=(0,2),(2,0); \\
                       10-d, & (p,q)=(1,1); \\
                       0, & \hbox{otherwise.}
\end{array}\right.
$$

\end{proposition}

\begin{proof}
Proposition \ref{prop-fpq} and Lemma \ref{proposition: del Pezzo f} give
$$f^{p,0}(Y,w)=h^p\left(Z,\Omega^0_Z(\logdiv D,f)\right)=h^p\left(Z,\omega_Z\right)=
\left\{\begin{array}{ll}
0, & \hbox{p=0,1;} \\
1, & \hbox{p=2,}
\end{array}
\right.
$$
$$f^{p,1}(Y,w)=h^p\left(Z,\Omega^1_Z(\logdiv D,f)\right)=
h^p\left(Z,\Omega^1_Z(\logdiv D)(-D)\right)=
\left\{ \begin{array}{ll}
                       0, & \hbox{p=0,2;} \\
                       10-d, & \hbox{p=1,}
                     \end{array}
                   \right.
                   $$
and
$$f^{p,2}(Y,w)=h^p\left(Z,\Omega^2_Z(\logdiv D,f)\right)=h^p\left(Z,\cO_Z\right)=\left\{ \begin{array}{ll}
1, & \hbox{i=0;} \\
0, & \hbox{i=1,2.}
\end{array}
\right.
$$
\end{proof}

This proposition computes the left hand side of equalities from Theorem~\ref{theorem:main}(i) and Theorem~\ref{theorem:main}(ii), and together with Proposition \ref{proposition:h-numbers} completes the
proof of parts (i) and (iii) of Theorem~\ref{theorem:main}. 

\section{End of proof of Theorem \ref{theorem:main} and discussion}
\label{section:main}

In this section we keep the notation from Section~\ref{section:LG-Hodge-numbers-for-surfaces}.
Studying elliptic surfaces in Section~\ref{section:LG-Hodge-numbers-for-surfaces} is motivated by Mirror Symmetry constructions
for del Pezzo surfaces from~\cite{AKO06}. The authors prove there ``a half'' of HMS conjecture for del Pezzo surfaces.
More precise, they prove that for a general del Pezzo surface $S_d$ of degree $d$, $1\leq d\leq 9$, obtained
by blow up of $\PP^2$ in $9-d$ general points there exist a complexified symplectic form $\omega _Y$ on $(Y,w)$, where~$(Y,w)$ has $12-d$
nodal singular fibers, and that $Y$ can be compactified to $Z$ for which $D$ is a wheel of $d$ curves, such that
\begin{equation}\label{HMS-equivalence}
D^b(\coh S_d)\simeq FS((Y,w),\omega_Y).
\end{equation}
We call $(Y,w)$ {\it a Landau--Ginzburg model for $S_d$}.
We allow the case $d=0$ as well; in this case~$(Y,w)$ is a Landau--Ginzburg model for $\PP^2$ blown up in $9$ intersection points
of two elliptic curves, see~\cite{AKO06}. The equivalence~\eqref{HMS-equivalence} holds in this case as well.

\begin{remark}
\label{remark: Hodge for del Pezzo}
The description of del Pezzo surface $X$ of degree $d$ as a blow up of $\PP^2$ gives the following equalities:
\begin{equation*}
h^{p,q}(X)=\left\{\begin{array}{rl}
                       1, & (p,q)=(0,2),(2,0); \\
                       10-d, & (p,q)=(1,1); \\
                       0, & \hbox{otherwise.}
\end{array}\right.
\end{equation*}
\end{remark}

This remark, together with Proposition~\ref{proposition:h-numbers} and Proposition~\ref{proposition:f-numbers}, provides a proof of
part (ii) of Theorem~\ref{theorem:main} and thus completes the proof of this theorem.
In other words, Conjecture~\ref{conj-2} and ``a half'' of Conjecture~\ref{conj-1} hold for (mirrors of) del Pezzo surfaces.

\begin{remark}
\label{remark:i-numbers}
The second part of Conjecture~\ref{conj-1} does not hold already for Landau--Ginzburg model $(Y,w)$ for $\PP^2$. Indeed, one has
$h^{0,0}(Y,w)=h^{1,1}(Y,w)=h^{2,2}(Y,w)=1$. However the Landau--Ginzburg model $(Y,w)$ has exactly three singular fibers, and the singular set of
these fibers is a single node.
Hence
the numbers $i^{p,q}(Y,w)$ are integers divisible by $3$.
\end{remark}

Another approach to Landau--Ginzburg models is given by a notion of {\it toric Landau--Ginzburg models}.
(Basically it is related to another ``arrow'' $A$--$B$ of HMS complimentary to the equivalence~\eqref{eq-of-cat}.)
A toric Landau--Ginzburg model for a Fano variety $X$ is a Laurent polynomial satisfying particular
conditions: its periods are related to Gromov--Witten invariants for $X$ in a specific way, it admits a Calabi--Yau compactification,
and it is related to some toric degeneration for $X$. For precise definitions and more details see in~\cite{Prz08} and~\cite{Prz13}.
Toric Landau--Ginzburg models are known for smooth toric varieties, Fano threefolds, complete intersections in projective spaces
or Grassmannians (see~\cite{Gi97},~\cite{Prz13},~\cite{ILP13},~\cite{CCGGK12},~\cite{PSh14a},~\cite{PSh14b},~\cite{PSh15b}).
For del Pezzo surfaces of degree greater than two toric Landau--Ginzburg model is a Laurent polynomial with support on an integral polygon
having exactly one strictly internal integral point; coefficients of the polynomial are determined by a symplectic form chosen on the del Pezzo surface, see~\cite{Prz16}.
There exist natural tame compactifications of these toric Landau--Ginzburg model, and one can see that
Theorem~\ref{theorem:main} holds for them.
In particular this means that Theorem~\ref{theorem:main} holds for a quadric surface,
which is not a blow up of $\PP^2$, and for Landau--Ginzburg models with singular fibers whose singularities are more complicated than  just a simple node, as opposed to the case of~\cite{AKO06}.

One more geometrical output of Conjecture~\ref{conj-1} and Conjecture~\ref{conj-2} is the following.

\begin{conjecture}[{\cite[Conjecture 1.1]{PSh15b}, see also~\cite{GKR12}}]
Let $X$ be a Fano variety of dimension $n$. Let $f_X$ be its toric Landau--Ginzburg model corresponding to an anticanonical symplectic
form on $X$.
Let $k_{f_X}$ be a number of all components of reducible fibers (without multiplicities) of a (fiberwise) Calabi--Yau compactification
for $f_X$ minus the number of reducible fibers. One has
$$
h^{1,n-1}(X)=k_{f_X}.
$$
\end{conjecture}
This conjecture for del Pezzo surfaces follows from the construction of compactifications of toric Landau--Ginzburg models;
it is proven for Fano threefolds of rank one (see~\cite{Prz13}) and for complete intersections (see~\cite{PSh15b}).

\end{document}